\documentclass[reqno,english,11pt]{amsart}
\usepackage[utf8]{inputenc}
\usepackage{amsmath,amssymb,amsthm,mathrsfs,color,times,textcomp,yfonts,mathtools,cases}
\usepackage[textheight=615pt, textwidth=360pt, left=0.8in, right=0.8in, top=0.8in, bottom=0.8in]{geometry}

\allowdisplaybreaks[4]

\usepackage{bm}

\usepackage[T1]{fontenc} 

\newcommand{\rme}{\mathrm{e}}
\newcommand{\rmd}{\mathrm{d}}

\usepackage{cite} 
\usepackage{subcaption}
\usepackage[normalem]{ulem}
\usepackage[export]{adjustbox}
\usepackage{esint}
\usepackage{xcolor}
\usepackage{array}
\usepackage[colorlinks=true]{hyperref}
\hypersetup{urlcolor=blue, citecolor=blue, linkcolor=red}

\hypersetup{
	colorlinks=true,
	linkcolor=red
}

\usepackage[square,numbers]{natbib}

\usepackage{float}
\usepackage{ulem}
\usepackage{syntonly}
\usepackage{mathtools}
\usepackage{bm}
\usepackage{soul}
\usepackage{amsfonts,amsmath,latexsym,verbatim,amscd,mathrsfs,color,array}
\usepackage[colorlinks=true]{hyperref}

\usepackage{amsmath,amssymb,amsthm,amsfonts,graphicx,color}
\usepackage{amssymb}
\usepackage{epstopdf}

\newcommand{\R}{ \mathbb{R}}

\theoremstyle{plain}
\newtheorem{theorem}{Theorem}[section]
\newtheorem{lemma}[theorem]{Lemma}

\newtheorem{corollary}[theorem]{Corollary}
\newtheorem{remark}{Remark}[theorem]
\newcommand{\bremark}{\begin{remark} \em}
	\newcommand{\eremark}{\end{remark} }

\numberwithin{equation}{section}

\begin{document}
	
	\title[Critical surface for two component BEC system]{Critical surface for two component Bose-Einstein Condensates}
\author[W. Peng]{Wenshuai Peng}
\author[X. Zeng]{Xiaoyu Zeng}
\author[Q. Zhang]{Qidi Zhang}
\author[H. Zhou]{Huan-Song Zhou}

\address[W. Peng]{
	 Center for Mathematical Sciences,
	Wuhan University of Technology
	\newline \indent Wuhan 430070, P.R. China}
\email{\href{mailto:wspeng@whut.edu.cn}{wspeng@whut.edu.cn}}

\address[X. Zeng]{ Center for Mathematical Sciences,
	Wuhan University of Technology
	\newline \indent Wuhan 430070, P.R. China
}
\email{\href{mailto:xyzeng@whut.edu.cn}{xyzeng@whut.edu.cn}}

\address[Q. Zhang]{ 
	Institut de Math\'{e}matiques de Jussieu, Sorbonne Universit\'{e}, Universit\'{e} Paris Cit\'{e}, 4 place, Jussieu, 75005 Paris, France
}
\email{\href{mailto:qzhang@imj-prg.fr}{qzhang@imj-prg.fr}}

\address[H. Zhou]{ Center for Mathematical Sciences,
	Wuhan University of Technology
	\newline \indent Wuhan 430070, P.R. China
}
\email{\href{mailto:hszhou@whut.edu.cn}{hszhou@whut.edu.cn}}


	
	\begin{abstract}

We investigate the existence of ground states for two-component Bose--Einstein condensates with intraspecies interactions $a_1, a_2 \in  (0, a^*)$ and interspecies interaction $\beta>0$. By carefully investigating an associated auxiliary minimization problem, we prove the existence of a unique, continuous, critical surface $\gamma = \gamma(a_1, a_2)$ between $\beta_*:= \sqrt{(a^* - a_1) (a^* - a_2)}$ and $\beta^*:= a^* - \frac{a_1 + a_2}{2}$. This surface provides a complete classification for the existence of minimizers: ground states exist for $\beta \in (0,\gamma)$ and do not exist for $\beta > \gamma$. Additionally, when the trapping potentials are continuous at a common zero point, there is no ground state for $\beta = \gamma$. Our results close the gap case $\beta \in [\beta_*, \beta^*]$ when $a_1\ne a_2$ in literature, significantly extending the results in \cite{Bao-Cai-1,guoBlowupSolutionsTwo2017a}.

\end{abstract}
\maketitle

\section{Introduction and main results}
	
The theoretical study of Bose--Einstein condensates (BECs) has attracted considerable attention over the past several decades, driven by both fundamental interest in quantum many-body physics and the rapid development of experimental techniques for ultracold atomic gases \cite{PhysRevLett.75.1687, PARKINS19981}. In particular, two-component (or binary) BECs have been widely investigated, as they exhibit a wealth of phenomena that are absent in single-component condensates, including phase separation, domain formation, and richer symmetry-breaking patterns \cite{PhysRevLett.81.1543, PhysRevLett.78.586, PhysRevLett.80.2027}. These features have motivated a large body of mathematical research devoted to understanding the ground states, dynamical properties, and phase transitions of binary condensates.

A standard theoretical framework for describing the ground states of two-component BECs in $\mathbb{R}^2$ is provided by the Gross--Pitaevskii (GP) theory \cite{hoBinaryMixturesBose1996}. Specifically, the ground states can be characterized as minimizers of the following constrained minimization problem
\begin{equation}\label{eq-minimizer-problem-e}
	 \mathfrak{e}(a_1, a_2, \beta) := \inf\limits_{(u_1, u_2) \in \mathcal{M}}\mathcal{E}_{a_1, a_2, \beta}(u_1, u_2).
\end{equation}
Here, $\mathcal{E}_{a_1, a_2, \beta}(\cdot)$ is the well-known Gross--Pitaevskii functional defined as
\begin{equation}\label{th-26Sep18-4}
	\mathcal{E}_{a_1, a_2, \beta}(u_1, u_2) := \sum_{i=1}^{2} \Big( \| \nabla u_i\|_2^2 + \int_{\R^2} V_i(x) |u_i(x)|^2 \rmd x - \frac{a_i}{2} \| u_i\|_4^4\Big) -\beta \| u_1 u_2\|_2^2.
\end{equation}	
The mass constraint $\mathcal{M}$ is defined as
\begin{equation}\label{thre-26Sep12-1} 
	\mathcal{M} := \big\{ (u_1, u_2) \in \mathcal{X} \mid \| u_1\|_2^2 = \| u_2\|_2^2 = 1 \big\}
\end{equation}
with the Hilbert space $\mathcal{X} := \mathcal{H}_1 \times \mathcal{H}_2$, and each $\mathcal{H}_i$, $i=1,2$, defined as
\begin{equation}
\mathcal{H}_i := \big\{ u \in H^1(\R^2) \mid \| u\|_{\mathcal{H}_i} < \infty \big\}
\mbox{ with the norm }
		\| u\|_{\mathcal{H}_i} := \Big( \| \nabla u\|_2^2 + \int_{\R^2} V_i(x) |u(x)|^2 \rmd x \Big)^{1/2}.
\end{equation}
$V_i(x)$, $i=1,2$, denotes the trapping potential and  satisfies
\begin{equation}\label{eq-cond-V}
V_i(x) \in L_{\rm{loc}}^{\infty} (\mathbb{R}^2), 
\quad
\lim\limits_{|x| \to \infty} V_i(x) = \infty,
\quad 
\inf_{x\in \mathbb{R}^2} V_i(x) = 0,
\quad i=1,2.
\end{equation}
We emphasize that, due to the constraints $\| u_1\|_2^2 = \| u_2\|_2^2 = 1$, the constrained minimization problem \eqref{eq-minimizer-problem-e} can be treated in the same way when both $V_1$ and $V_2$ are bounded from below in $\mathbb{R}^2$. Without loss of generality, we assume that $\inf_{x\in \mathbb{R}^2} V_i(x) = 0$, $i=1,2$.

The constants $a_1, a_2 \in \R$ represent the intraspecies interactions of the cold atoms within each component, and the constant $\beta \in \R$ denotes the interspecies interaction between the two components. 
Moreover, $a_i > 0$ ($i=1,2$) or $\beta > 0$ indicates that the corresponding interaction is attractive; otherwise, it is repulsive. 

If $(u_1,u_2)\in \mathcal{X}$ is a minimizer of \eqref{eq-minimizer-problem-e}, then there exist Euler--Lagrange multipliers $\mu_1, \mu_2 \in \R$, also called chemical potentials, such that $(u_1,u_2)$ satisfies the following elliptic system
\begin{equation*}
	\begin{cases}
		- \Delta u_1 + V_1 (x) u_1 = \mu_1 u_1 + a_1 u_1^3 + \beta u_2^2 u_1, \\[4pt]
		- \Delta u_2 + V_2 (x) u_2 = \mu_2 u_2 + a_2 u_2^3 + \beta u_1^2 u_2,
	\end{cases}
	\text{ in } \R^2.
\end{equation*}
	The minimization problem \eqref{eq-minimizer-problem-e} and its analogs have been extensively studied; for results on existence, uniqueness, and symmetry breaking, we refer the reader to the literature \cite{guoBlowupSolutionsTwo2017a, MR4847306, Bao-Cai-1}. 
    A key analytical tool in investigating the attainability of such minimization problems is the following Gagliardo--Nirenberg inequality \cite{weinstein1983b}:
	\begin{equation}\label{eq-GN-inequality}
		\| u \|_4^4 \le \frac{2}{\| Q\|_2^2} \| \nabla u\|_2^2 \| u\|_2^2,\quad \forall u \in H^1(\R^2),
	\end{equation} 
	where the equality is achieved by $u(x) = Q(|x|)$, and $Q$ is the unique positive radial solution of
	\begin{equation*}
		- \Delta u - u^3 + u = 0,\quad x \in \R^2,\ u \in H^1(\R^2).
	\end{equation*}
	Pohozaev's identity shows that $Q(|x|)$ satisfies
	\begin{equation}\label{eq-Q-property-1}
		a^*:= \| Q \|_2^2 = \| \nabla Q\|_2^2 = \frac{1}{2} \| Q\|_4^4.
	\end{equation}

	For the two-component problem, the existence and non-existence of minimizers depend critically on the parameters $a_1, a_2, \beta$. Following \cite{Bao-Cai-1,guoBlowupSolutionsTwo2017a}, we define
	\begin{equation}\label{def-beta*}
		\beta_* = \beta_*(a_1, a_2) := \sqrt{(a^* - a_1)(a^* - a_2)}, \qquad 
		\beta^* = \beta^*(a_1, a_2) := a^* - \frac{a_1 + a_2}{2}.
	\end{equation}

    Previous works have established a partial classification of the existence and nonexistence of minimizers for \eqref{eq-minimizer-problem-e}.
    More precisely, we have
    \begin{itemize}
        \item If $0 < a_1, a_2 < a^*$ and $\beta \in [0, \beta_*)$, then problem \eqref{eq-minimizer-problem-e} admits at least one minimizer. See \cite[Theorem 1.1 (i)]{guoBlowupSolutionsTwo2017a} and \cite[Theorem 2.6 (ii)]{Bao-Cai-1}.
        
        \item If $0 \le a_1, a_2 < a^*$ and $\beta < 0$, then problem \eqref{eq-minimizer-problem-e} admits at least one minimizer. See \cite[Theorem 1.1 (i)]{2018-Guo-Zeng-Zhou} and \cite[Theorem 2.6 (ii)]{Bao-Cai-1}.
        
        \item Assume that $a_1, a_2, \beta > 0$. Then if either $a_1 > a^*$, or $a_2 > a^*$, or $\beta > \beta^*$, the problem \eqref{eq-minimizer-problem-e} has no minimizer. See \cite[Theorem 1.1 (ii)]{guoBlowupSolutionsTwo2017a}.

        \item Assume that $a_1, a_2 > 0$ and $\beta < 0$. Then if either $a_1 > a^*$, or $a_2 > a^*$, or $a_1 = a_2 = a^*$, the problem \eqref{eq-minimizer-problem-e} has no minimizer. See \cite[Theorem 1.1 (ii)]{2018-Guo-Zeng-Zhou}.
    \end{itemize}
    However, the solvability of \eqref{eq-minimizer-problem-e} remains largely open in the parameter region
	\begin{equation}\label{eq-beta}
    0 < a_1, a_2 < a^*,
    \quad 
    \beta_* \le \beta \le \beta^*,
	\end{equation}
	except for two special cases discussed in \cite[Theorems 1.2, 1.3]{guoBlowupSolutionsTwo2017a}. 
    The first case is that when $a_1 \neq a_2$ satisfy $|a_1 - a_2| \le 2\beta_*$, \eqref{eq-minimizer-problem-e} possesses a minimizer provided $\beta \ge \beta_*$ but very close to $\beta_*$.
    The second case is that $a_1 = a_2$, for which $\beta_* = \beta^*$, then there is no minimizer when the potentials $V_1(x)$ and $V_2(x)$ share the same minimal points.

	In the remaining region described by \eqref{eq-beta}, it is very difficult to obtain the boundedness of minimizing sequences in $\mathcal{X}$. 
    Consequently, standard analytical tools---such as the concentration-compactness principle or blow-up techniques---cannot be directly applied to prove the existence or non-existence of minimizers for \eqref{eq-minimizer-problem-e}. 
    This difficulty constitutes the main open problem in the theoretical study of two-component BECs with attractive interactions.
	
	In this paper, we present new observations on the limiting behavior of minimizing sequences. We show that there exists a unique critical surface $\gamma(a_1,a_2)$, lying between the surfaces $\beta_*$ and $\beta^*$, which precisely determines the attainability of problem \eqref{eq-minimizer-problem-e}. More specifically, we prove that a minimizer exists if the parameter $\beta$ is below the critical surface $\gamma(a_1,a_2)$, and no minimizer exists above it. Consequently, we establish a \textbf{complete classification} on the existence of minimizers for problem \eqref{eq-minimizer-problem-e} with respect to the parameters $a_1, a_2, \beta$, and answer the aforementioned open problem.
	
	The main results are stated as follows.
\begin{theorem}\label{thm-main}
Suppose that $V_1(x), V_2(x)$ satisfy \eqref{eq-cond-V}, $0 < a_1, a_2 < a^*$, then there exists a unique critical surface $\gamma = \gamma(a_1, a_2)$ only depending on $a_1, a_2$ such that \eqref{eq-minimizer-problem-e} admits at least one minimizer when $\beta \in (0, \gamma)$, whereas no minimizer exists for \eqref{eq-minimizer-problem-e} when $\beta > \gamma$. Moreover, $\gamma$ is continuous about $a_1, a_2$, and
\begin{equation*}
\begin{cases}
        \gamma = a^* - a_1
        & \mbox{ if \ } a_1 = a_2,
        \\
        \gamma \in (\sqrt{(a^* - a_1)(a^* - a_2)}, a^* - \frac{a_1 + a_2}{2}) & \mbox{ if \ } a_1 \ne a_2.
        \end{cases}
\end{equation*}
\end{theorem}

On the critical surface $\beta = \gamma$, the existence of minimizers depends on further properties of the potentials $V_1, V_2$.
\begin{theorem}\label{thm-main-2}
Suppose that $V_1(x), V_2(x)$ satisfy \eqref{eq-cond-V}, $0 < a_1, a_2 < a^*$, and $\beta=\gamma$ with $\gamma$ given in Theorem \ref{thm-main}, if there exists $x_0 \in \mathbb{R}^2$ such that
\begin{equation}\label{eq-V-same-roots}
V_1(x_0) = V_2(x_0) = 0,
\end{equation}
and $V_1(x), V_2(x)$ are continuous at $x_0$, then \eqref{eq-minimizer-problem-e} has no minimizer.
\end{theorem}

\begin{remark}
		If the mass constraint $\mathcal{M}$ \eqref{thre-26Sep12-1} is replaced by $\{(u_1, u_2) \in \mathcal{X} \mid \| u_1 \|_2^2 + \| u_2\|_2^2 = 1\}$, then the critical surface for problem \eqref{eq-minimizer-problem-e} becomes $\bar{\gamma} = a^* + \sqrt{(a^* - a_1) (a^* - a_2)}$. See \cite[Theorem 1.1 and Theorem 1.2]{guoGroundStatesTwocomponent2019c} and \cite[Theorem 1.1]{pengGroundStatesTwoComponent2024} for details. Moreover, they also proved that there exists no minimizer on the critical surface ($\beta = \bar{\gamma}$) under certain conditions for potentials $V_1, V_2$. 
\end{remark}

	The proof of Theorem \ref{thm-main} is mainly inspired by \cite{guoBlowupSolutionsTwo2017a}.  We introduce the following auxiliary minimization problem:
	\begin{equation}\label{eq-minimizer-problem-O}
		\mathcal{O}(a_1, a_2, \beta) := \inf\limits_{u_i \in H^1(\mathbb{R}^2), \,  \| u_i\|_2^2 = 1, i=1,2, \,  \frac{a_1}{2} \| u_1\|_4^4 + \frac{a_2}{2} \| u_2\|_4^4 + \beta \| u_1 u_2\|_2^2 > 0} \mathcal{F}_{a_1, a_2, \beta}(u_1, u_2)
	\end{equation}
with the functional
	\begin{equation} 
		\mathcal{F}_{a_1, a_2, \beta}(u_1, u_2) := \frac{\| \nabla u_1\|_2^2 + \| \nabla u_2\|_2^2}{\frac{a_1}{2} \| u_1\|_4^4 + \frac{a_2}{2} \| u_2\|_4^4 + \beta \| u_1 u_2\|_2^2}.
	\end{equation}

By \cite[Proposition 1]{guoBlowupSolutionsTwo2017a},  \eqref{eq-minimizer-problem-e} admits at least one minimizer if $\mathcal{O}(a_1, a_2, \beta) > 1$, whereas no minimizers exist when $\mathcal{O}(a_1, a_2, \beta) < 1$. Therefore, the key step in proving Theorem \ref{thm-main} is to analyze the value of $\mathcal{O}(a_1, a_2, \beta)$. Indeed, in Theorem \ref{thm-gamma}, the critical surface $\gamma$ characterizes the existence of minimizers in the following way: if $0 < \beta < \gamma$, then $\mathcal{O}(a_1, a_2, \beta) > 1$, which implies that \eqref{eq-minimizer-problem-e} admits a minimizer; if $\beta > \gamma$, then $\mathcal{O}(a_1, a_2, \beta) < 1$, and consequently no minimizer exists. For Theorem \ref{thm-main-2}, the key point is to prove  $\mathfrak{e}(a_1, a_2, \gamma) = 0$.

This paper is organized as follows. In Section \ref{sec-2}, we give some properties of the auxiliary minimization problem \eqref{eq-minimizer-problem-O}. In Section \ref{sec-3}, we prove that when $0 < a_1 \neq a_2 < a^*$, then $\mathcal{O}(a_1, a_2, \beta_*) > 1$ and $\mathcal{O}(a_1, a_2, \beta^*) < 1$, which implies the critical surface $\gamma$ satisfying $\gamma \in (\beta_*, \beta^*)$. The case $a_1 = a_2 \in (0, a^{*})$ has already been treated in \cite{guoBlowupSolutionsTwo2017a}. Finally, we give the proof of Theorems \ref{thm-main} and \ref{thm-main-2}.

\section{Properties of auxiliary minimization problem}\label{sec-2}

Throughout this paper, when handling $\mathcal{F}_{a_1, a_2, \beta}(u_1, u_2)$, we always assume 
\begin{equation}\label{th-26Sep15-4}
a_1, a_2, \beta \ge 0, 
\ 
a_1 + a_2 + \beta >0,
\ 
u_1, u_2 \in H^1(\mathbb{R}^2), \ \frac{a_1}{2} \| u_1\|_4^4 + \frac{a_2}{2} \| u_2\|_4^4 + \beta \| u_1 u_2\|_2^2 > 0.
\end{equation}
If $\| u_i\|_2^2 \le 1, i=1,2$, by \eqref{eq-GN-inequality} and H\"older inequality, we have 
\begin{equation*}
\frac{a_1}{2} \| u_1\|_4^4 + \frac{a_2}{2} \| u_2\|_4^4 + \beta \| u_1 u_2\|_2^2
\le 
\frac{a_1 + \beta}{a^{*}} \| \nabla u_1\|_2^2 + \frac{a_2 + \beta}{a^{*}} \| \nabla u_2\|_2^2,
\end{equation*}
and then
\begin{equation}\label{thre-26Sep15-3} 
\mathcal{F}_{a_1, a_2, \beta}(u_1, u_2) \ge
\min\Big\{ \frac{a^*}{a_1 + \beta}, \frac{a^*}{a_2 + \beta}  \Big\},
\end{equation}
where $\min\{ \tfrac{a^*}{a_1 + \beta}, \tfrac{a^*}{a_2 + \beta} \} \in (0, \infty)$ by the assumption \eqref{th-26Sep15-4}. In particular, it follows that
\begin{equation}\label{26July8-6}
\mathcal{O}(a_1, a_2, \beta) \in (0,\infty).
\end{equation}

In this section, we give some important properties about the auxiliary minimization problem $\mathcal{O}(\cdot)$ in \eqref{eq-minimizer-problem-O}.
	\begin{lemma}[{\cite[Proposition 1]{guoBlowupSolutionsTwo2017a}}] \label{lem-existence-1}
		Suppose that $V_1(x), V_2(x)$ satisfy \eqref{eq-cond-V}, and $a_1, a_2, \beta > 0$. Then
		\begin{enumerate}
			\item[(i)] \eqref{eq-minimizer-problem-e} has at least one minimizer if $\mathcal{O}(a_1, a_2, \beta) > 1$.
			\item[(ii)] \eqref{eq-minimizer-problem-e} has no minimizer if $\mathcal{O}(a_1, a_2, \beta) < 1$.
		\end{enumerate}
	\end{lemma}

	\begin{lemma}[{\cite[Lemma 2.2, the proof of Theorem 1.1]{guoBlowupSolutionsTwo2017a}}]\label{lem-O-properties}
		Suppose $a_1, a_2, \beta > 0$, then
		\begin{enumerate}
			\item[(i)] $\mathcal{O}(\cdot)$ is locally Lipschitz continuous.
			\item[(ii)] If $0 < a_1,a_2 < a^*$ and $\beta < \beta_*$, then $\mathcal{O}(a_1, a_2, \beta) > 1$.
			\item[(iii)] If $a_1 > a^*$ or $a_2 > a^*$ or $\beta > \beta^*$, then $\mathcal{O}(a_1, a_2, \beta) < 1$.
			\item[(iv)] $\mathcal{O}(a_1, a_2, \beta)$ is non-increasing in $\beta > 0$.
		\end{enumerate}
	\end{lemma}

	\begin{remark}\label{26July9-1-rmk}
		\begin{enumerate}
			\item[(a)] From (i)-(iii) of Lemma \ref{lem-O-properties}, one sees that $\mathcal{O}(a_1, a_2, \beta_*) \ge 1$ and $\mathcal{O}(a_1, a_2, \beta^*) \le 1$.
			\item[(b)] For any fixed $0<a_1,a_2<a^*$, then there exists an interval $[\beta_1, \beta_2] \subset [\beta_*, \beta^*]$, such that $\mathcal{O}(a_1, a_2, \beta) = 1$ for all $\beta \in [\beta_1, \beta_2]$. We will show that it must hold that  $\beta_1=\beta_2$ in Theorem \ref{thm-gamma}.         
			\item[(c)] When $a_1 = a_2 = a\in(0,a^*)$, we have $\beta_*=\beta^*$. One then deduces from Lemma \ref{lem-O-properties}  that $\mathcal{O}(a_1, a_2, \beta_*) = \mathcal{F}_{a_1, a_2, \beta_*}(\frac{Q}{\sqrt{a^*}}, \frac{Q}{\sqrt{a^*}}) = 1$.
			\item[(d)] The last property (iv) is obvious by the definition of the functional $\mathcal{F}_{a_1, a_2, \beta}(u_1, u_2)$.
		\end{enumerate}
	\end{remark}

For any $C > 0$, we call $C f(C x)$ as the {\it{$L^2$-preserving scaling}} of $f$ since $\| C f(C x) \|_{2} = \| f \|_{2}$. Direct calculation shows
\begin{equation}\label{26July6-1}
\mathcal{F}_{a_1, a_2, \beta}(C u_{1}(C x), C u_{2}(C x))
=
\mathcal{F}_{a_1, a_2, \beta}(u_1, u_2).
\end{equation}
For any $f \in H^1(\mathbb{R}^2)$, by diamagnetic inequality, $|\nabla|f|| \le |\nabla f|$ a.e. in $\mathbb{R}^2$ (See \cite[Theorem 7.21]{liebAnalysis2001}). Thus, it holds that $\mathcal{F}_{a_1, a_2, \beta}(|u_1|, |u_2|) \le \mathcal{F}_{a_1, a_2, \beta}(u_1, u_2)$. Moreover, for any non-negative functions $u_1, u_2 \in H^1(\mathbb{R}^2)$, the symmetric-decreasing rearrangement $u_1^{*}, u_2^{*}$ satisfy $\| u_i \|_{p} = \| u_i^{*} \|_{p}$, $p\in [2,\infty)$, $\| \nabla u_i \|_{2} \ge \| \nabla u_i^{*} \|_{2}$, $i=1,2$ by \cite[Lemma 7.17]{liebAnalysis2001}, and $\| u_1 u_2\|_{2}^2 \le \| u_1^{*} u_2^{*} \|_{2}^2$ by \cite[(v) in p. 81 and Theorem 3.4]{liebAnalysis2001}. Then
$\mathcal{F}_{a_1, a_2, \beta}(u_1^{*}, u_2^{*}) \le \mathcal{F}_{a_1, a_2, \beta}(u_1, u_2)$. Combining the property \eqref{26July6-1}, we always assume that any minimizing sequence $\{ (u_{1n}, u_{2n})\}_{n\ge 1}$ about $\mathcal{F}_{a_1, a_2, \beta}(\cdot)$ satisfies 
\begin{equation}\label{thre-26Sep13-4}
\begin{aligned}
&
u_{in} \in H^1_r(\R^2), 
\quad
u_{in} \ge 0,
\quad
\mbox{$u_{in}$ is radially symmetric and non-increasing,} \quad i=1,2,
\\
&
\| \nabla u_{1n}\|_2^2 + \| \nabla u_{2n}\|_2^2 \stackrel{\eqref{26July6-1}}{=} 1,
\quad
\frac{a_1}{2} \| u_{1n}\|_4^4 + \frac{a_2}{2} \| u_{2n}\|_4^4 + \beta \| u_{1n} u_{2n}\|_2^2 > 0.
\end{aligned}
\end{equation}

We analyze the attainability of \eqref{eq-minimizer-problem-O} in the following lemma.
\begin{lemma}
    \label{lem-existence-O}
    Assume that $a_1, a_2, \beta \ge 0$, $a_1 + a_2 + \beta > 0$, and
    \begin{equation}\label{eq-energy-condition}
    \mathcal{O}(a_1,a_2,\beta)
    <
    \min\Big\{
        \frac{a^*}{a_1},
        \frac{a^*}{a_2}
    \Big\}.
    \end{equation}
Then the infimum in \eqref{eq-minimizer-problem-O} is attained by a pair $(u_1, u_2)$. Moreover, we can make that for $i=1,2$, $u_i \ge 0$ is smooth, radially symmetric, non-increasing, and there exist constants $C, \sigma > 0$ such that
\begin{equation}\label{eq-exponential-decay}
|u_i| + |\nabla u_i| \le C \rme^{- \sigma |x|}
\mbox{ \ in \ } \mathbb{R}^2.
\end{equation}
\end{lemma}

\begin{proof}
    Define the relaxed constraint
    \begin{equation*}
    \mathcal{S}
        :=
        \Big\{
            (u_1,u_2) \ \big| \ u_i \in H^1(\mathbb{R}^2), 
            \| u_i \|_2^2\leq 1, i=1,2,
            \quad
            \frac{a_1}{2} \| u_1\|_4^4 + \frac{a_2}{2} \| u_2\|_4^4 + \beta \| u_1 u_2\|_2^2>0
        \Big\},
    \end{equation*}
    and the relaxed minimizing problem
    \begin{equation}\label{eq-minimizer-problem-barO}
        \overline{\mathcal{O}}(a_1, a_2, \beta) := \inf\limits_{(u_1, u_2) \in \mathcal{S}} \mathcal{F}_{a_1, a_2, \beta}(u_1, u_2).
    \end{equation}
By \eqref{thre-26Sep15-3}, we have $\overline{\mathcal{O}}(a_1, a_2, \beta) \in (0, \infty)$.

For brevity, in this proof, we use $\overline{\mathcal{O}}, \mathcal{O}$ to denote $\overline{\mathcal{O}}(a_1, a_2, \beta), \mathcal{O}(a_1, a_2, \beta)$ respectively.
  
    First, we claim $\overline{\mathcal{O}} = \mathcal{O}$. Obviously, $\overline{\mathcal{O}} \le \mathcal{O}$. It remains to prove $\overline{\mathcal{O}} \ge \mathcal{O}$. Let $(v_1, v_2) \in C_c^{\infty}(\R^2) \cap \mathcal{S}$ and set $m_i := \| v_i \|_2^2 \le 1$, $i=1,2$. 
    Choose $\eta \in C_c^{\infty}(\R^2)$ such that $\| \eta\|_2^2 = 1$. Define $\eta_{R,y}(x) := \frac{1}{R} \eta(\frac{x-y}{R})$. For any $R \ge 1$, there exist $y_{1,R}, y_{2,R} \in \mathbb{R}^2$ such that
    \begin{equation} \label{eq-disjoint}
        \text{supp } v_{i} \cap \text{supp }\eta_{R, y_{j,R}} = \emptyset \mbox{ \ for \ } i,j\in \{1,2\}, \quad \text{supp } \eta_{R, y_{1,R}}\cap \text{supp } \eta_{R, y_{2,R}} = \emptyset.
    \end{equation}
    Define $\tilde{v}_{i, R} := v_i + \sqrt{1 - m_i} \eta_{R, y_{i, R}}$. Since $\| \eta_{R,y} \|_2^2 = 1$, $\| \nabla\eta_{R,y} \|_2^2
    =R^{-2}\| \nabla\eta \|_2^2$, $\| \eta_{R,y}\|_4^4
    =R^{-2}\| \eta\|_4^4$, by \eqref{eq-disjoint}, we have
\begin{equation*}
\begin{aligned}
&
\| \tilde{v}_{i, R} \|_2^2 = 1, i=1,2,
\quad
\| \nabla \tilde{v}_{1, R}\|_2^2 + \| \nabla \tilde{v}_{2, R}\|_2^2 =
        \| \nabla v_1 \|_2^2 + \| \nabla v_2 \|_2^2 + O(R^{-2}),
\\
&
        \frac{a_1}{2} \| \tilde{v}_{1, R} \|_4^4 + \frac{a_2}{2} \| \tilde{v}_{2, R} \|_4^4 + \beta \| \tilde{v}_{1, R} \tilde{v}_{2, R} \|_2^2 
        \\
        = \ & \frac{a_1}{2} 
        \big[
        \| v_1 \|_4^4 + (1-m_1)^2 \| \eta_{R, y_{1,R}} \|_4^4
        \big]
        + \frac{a_2}{2} 
        \big[ 
        \| v_2 \|_4^4
        + (1-m_2)^2 \| \eta_{R, y_{2,R}} \|_4^4 
        \big] + \beta \| v_1 v_2 \|_2^2
        \\
        = \ & \frac{a_1}{2} \| v_1 \|_4^4 + \frac{a_2}{2} \| v_2 \|_4^4 + \beta \| v_1 v_2 \|_2^2 + O(R^{-2}) >0,
\end{aligned}
\end{equation*}
where for the last `` $>$ '', we use $(v_1, v_2) \in \mathcal{S}$ and require $R$ sufficiently large. Thus,
    \begin{equation*}
        \mathcal{O} \le \liminf\limits_{R \to \infty} \mathcal{F}_{a_1, a_2, \beta}(\tilde{v}_{1, R}, \tilde{v}_{2,R}) = \mathcal{F}_{a_1, a_2, \beta}(v_1, v_2).
    \end{equation*}
    Since $C_c^{\infty}(\R^2)$ is dense in $H^1(\R^2)$, taking the infimum over $\mathcal{S}$ yields $\overline{\mathcal{O}} \ge \mathcal{O}$. Thus, $\overline{\mathcal{O}} = \mathcal{O}$.

    Next, choose a minimizing sequence $\{ (u_{1n}, u_{2n})\}_{n \ge 1}$ of $\overline{\mathcal{O}}$ satisfying \eqref{thre-26Sep13-4} and $\| u_{in} \|_2^2 \le 1$, $i=1,2$. Since the  embedding $H^1_r(\R^2) \hookrightarrow L^p(\R^2)$ is compact for $p\in (2,\infty)$ (see \cite[Compactness Lemma 2]{straussExistenceSolitaryWaves1977} or \cite[Corollary 1.26]{willemMinimaxTheorems1996}), then up to a subsequence, there exist non-negative, radially symmetric, non-increasing functions $u_1, u_2 \in H_r^1(\mathbb{R}^2)$ such that for $i=1,2$,
	\begin{equation}\label{th-26Sep15-5}
        \begin{aligned}
            &
			u_{in} \rightharpoonup u_i \mbox{ \ in \ } H_r^1(\R^2),
            \quad
            \nabla u_{in} \rightharpoonup \nabla u_i \mbox{ \ in \ }
            L^2(\mathbb{R}^2),
            \quad
            u_{in} \rightharpoonup u_i \mbox{ \ in \ }
            L^2(\mathbb{R}^2),
            \\
            & 
			\| u_{in} - u_i \|_{4} \to 0, 
            \quad
            \| u_{1n} u_{2n} - u_1 u_2 \|_{2} \to 0
            \mbox{ \ as \ } n \to \infty,
            \quad
            \|\nabla u_i\|_{2}^2 \le \liminf_{n \to \infty} \|\nabla u_{in}\|_{2}^2.
        \end{aligned}
	\end{equation}
By $\overline{\mathcal{O}} \in (0, \infty)$ and $\| \nabla u_{1n}\|_2^2 + \| \nabla u_{2n}\|_2^2 = 1$ in \eqref{thre-26Sep13-4}, there exist a sufficiently large constant $N_0$, and
constants $C_2 \ge C_1 >0$ such that
\begin{equation*}
C_1 \le \frac{a_1}{2} \| u_{1n}\|_4^4 + \frac{a_2}{2} \| u_{2n}\|_4^4 + \beta \| u_{1n} u_{2n}\|_2^2 \le C_2
\mbox{ \ for \ } n>N_0,
\end{equation*}
which implies $C_1 \le \frac{a_1}{2} \| u_{1}\|_4^4 + \frac{a_2}{2} \| u_{2}\|_4^4 + \beta \| u_{1} u_{2}\|_2^2 \le C_2$. Thus, $(u_1, u_2)$ attains $\overline{\mathcal{O}}$.

To show that $(u_1, u_2)$ is actually the minimizer of $\mathcal{O}$, it suffices to prove $\| u_1 \|_2^2 = \| u_2\|_2^2 = 1$. If $\| u_1 \|_2^2 < 1$, then for any $\varphi\in C_c^\infty(\R^2)$, there exists a sufficiently small constant $\epsilon>0$ depending on the choice of $\varphi$, such that for any $t$ satisfying $|t| <\epsilon$, the pair $(u_1+t\varphi,u_2)$ remains in $\mathcal{S}$. Differentiating $\mathcal{F}_{a_1, a_2,\beta}(u_1 + t\varphi, u_2)$ at $t=0$ gives
    \begin{equation*}
        \Delta (- u_1)
        =
        \overline{\mathcal{O}} ( a_1 u_1^3+\beta u_1 u_2^2 )
    \end{equation*}
in the weak sense. Combining the Euler-Lagrange equation of $u_2$ depending on $\| u_2\|_2^2 = 1$ or $< 1$, similar to \eqref{th-26Sep17-1} below, we have $u_1, u_2 \in C^{\infty}(\mathbb{R}^2)$. Since $\overline{\mathcal{O}} ( a_1 u_1^3+\beta u_1 u_2^2 )
        \ge 0$, $-u_1$ is bounded above and $\|u_1\|_{2} \le 1$, by Liouville's theorem for subharmonic functions in $\mathbb{R}^2$ (see \cite[Problem 2.14]{gilbargEllipticPartialDifferential2001} for example), we have $u_1 \equiv 0$. Then $\overline{\mathcal{O}} = \tfrac{2\| \nabla u_2 \|_2^2}{a_2 \| u_2 \|_4^4}$ and $a_2 > 0$. By \eqref{eq-GN-inequality}, $\| u_2\|_2^2 \le 1$, we deduce $\overline{\mathcal{O}} \ge \frac{a^*}{a_2}$, which contradicts with $\overline{\mathcal{O}} = \mathcal{O} < \frac{a^*}{a_2}$ in \eqref{eq-energy-condition}. Thus, $\| u_1 \|_2^2=1$. By $\mathcal{O} < \frac{a^*}{a_1}$ in \eqref{eq-energy-condition} and similar argument, we also have $\| u_2 \|_2^2=1$. Thus, we find a minimizer $(u_1, u_2)$ of $\mathcal{O}$.

Moreover, for $(u_1, u_2)$ as the minimizer of $\mathcal{O}$, it holds the following Euler-Lagrange equation 
    \begin{equation}\label{th-26Sep17-1}
    \begin{cases}
        -\Delta u_1 + \lambda_1 u_1 = \mathcal{O} (a_1 u_1^3 + \beta u_1 u_2^2), 
        \\
        -\Delta u_2 + \lambda_2 u_2 = \mathcal{O} (a_2 u_2^3 + \beta u_1^2 u_2),
    \end{cases}
        \mbox{ \ in \ } \mathbb{R}^2
    \end{equation} 
with the Lagrange multipliers $\lambda_1, \lambda_2$. For $i=1,2$, we have $\lambda_i> 0$, otherwise we have $u_i \equiv 0$ by Liouville's theorem for subharmonic functions in $\mathbb{R}^2$. Since $u_1, u_2 \in H^1(\mathbb{R}^2)$, by the positive integer powers of the nonlinear term in \eqref{th-26Sep17-1} and elliptic regularity theory, we have $u_1, u_2 \in C^{\infty}(\mathbb{R}^2)$. Since $|u_i| \to 0 $ as $|x| \to \infty$, there exists a large constant $R_1 > 0$ such that
    \begin{equation*}
    \begin{cases}
        -\Delta u_1 + \frac{\lambda_1}{2} u_1 \le 0,\\
        -\Delta u_2 + \frac{\lambda_2}{2} u_2 \le 0,
    \end{cases}
    \mbox{ \ for \ } |x| \ge R_1.
    \end{equation*}
For $i=1,2$, we set $\sigma_i = \sqrt{\lambda_i/2}$ and $w_i(x) = C_i \rme^{-\sigma_i ( |x| - R_1 )}$ with a sufficiently large constant $C_i > 0$. So $w_i$ is a barrier function of $u_i$ in $|x| \ge R_1$. It follows that $u_i(x) \le w_i(x) = C_i \rme^{\sigma_i R_1} \rme^{-\sigma_i |x|}$ for $|x| \ge R_1$. Then $u_i \lesssim \rme^{-\sigma_i |x|}$ in $\mathbb{R}^2$. So we have $|\mathcal{O} (a_1 u_1^3 + \beta u_1 u_2^2) - \lambda_1 u_1| \lesssim \rme^{-\sigma_1 |x|} $, $|\mathcal{O} (a_2 u_2^3 + \beta u_1^2 u_2) - \lambda_2 u_2| \lesssim \rme^{-\sigma_2 |x|} $. Using gradient estimate in $B(x,2)$ for $|x| \ge 3$, we have $|\nabla u_i| \lesssim \rme^{-\sigma_i |x|}$ in $\mathbb{R}^2$, $i=1,2$.
\end{proof}

\begin{corollary}
\label{cor-decrease}
    Under the assumption in Lemma \ref{lem-existence-O}, then for $\beta_1 > \beta$, we have $\mathcal{O}(a_1, a_2, \beta_1) < \mathcal{O}(a_1, a_2, \beta) $.
\end{corollary}
\begin{proof}
By Lemma \ref{lem-existence-O},
$\mathcal{O}(a_1,a_2,\beta)$ is attained by a pair $(u_1,u_2)$ and $\| u_1 u_2\|_2 >0$. By $\beta_1>\beta$,
\begin{equation*}
\mathcal{O}(a_1,a_2,\beta)
    =\mathcal{F}_{a_1,a_2,\beta}(u_1,u_2)
    >\mathcal{F}_{a_1,a_2,\beta_1}(u_1,u_2)
    \ge \mathcal{O}(a_1,a_2,\beta_1).  
\end{equation*}
\end{proof}
	
\section{Existence of critical surface and non-existence result on critical surface}\label{sec-3}
In this section, we prove the two main results in this paper.
In order to prove Theorem \ref{thm-main}, 
first, we estimate the value of $\mathcal{O}(\cdot)$ on the lower surface $\beta_* = \beta_*(a_1, a_2)$ and the upper surface $\beta^* = \beta^*(a_1, a_2)$.

	\begin{lemma}\label{lem-O-energy-estimate}
		Suppose $0 < a_1 \neq a_2 < a^*$. Then $\mathcal{O}(a_1, a_2, \beta_*) > 1$ and $\mathcal{O}(a_1, a_2, \beta^*) < 1$.
	\end{lemma}
	
	\begin{proof}
		\textbf{Case 1: The lower surface $\beta_* = \beta_*(a_1, a_2)$.} By Remark \ref{26July9-1-rmk}, $\mathcal{O}(a_1, a_2, \beta_*) \ge 1$. Assume, for contradiction, that $\mathcal{O}(a_1, a_2, \beta_*) = 1$. Similar to \eqref{th-26Sep15-5}, there exists a minimizing sequence $\{ (u_{1n}, u_{2n})\}_{n \ge 1}$ of $\mathcal{O}(a_1, a_2, \beta_*) = 1$ satisfying \eqref{thre-26Sep13-4}, and there exist non-negative, radially symmetric, non-increasing functions $u_1, u_2 \in H_r^1(\mathbb{R}^2)$ such that for $i=1,2$,
\begin{equation}\label{eq-limit}
		\lim_{n\to \infty}	\mathcal{F}_{a_1, a_2, \beta_*}(u_{1n}, u_{2n}) = 1,
		\end{equation}
	\begin{equation}\label{eq-weak-convergence}
        \begin{aligned}
            &
			u_{in} \rightharpoonup u_i \mbox{ \ in \ } H_r^1(\R^2),
            \quad
            \nabla u_{in} \rightharpoonup \nabla u_i \mbox{ \ in \ }
            L^2(\mathbb{R}^2),
            \quad
            u_{in} \rightharpoonup u_i \mbox{ \ in \ }
            L^2(\mathbb{R}^2),
            \\
            & 
			\| u_{in} - u_i \|_{4} \to 0, 
            \quad
            \| u_{1n} u_{2n} - u_1 u_2 \|_{2} \to 0
            \mbox{ \ as \ } n \to \infty.
        \end{aligned}
	\end{equation}
By \eqref{eq-limit} and $\| \nabla u_{1n}\|_2^2 + \| \nabla u_{2n}\|_2^2 = 1$ in \eqref{thre-26Sep13-4}, there exist large constants $N_0>0$ and $C_1>1$ such that $\frac{a_1}{2} \| u_{1n}\|_4^4 + \frac{a_2}{2} \| u_{2n}\|_4^4 + \beta_{*} \| u_{1n} u_{2n}\|_2^2 \in (C_1^{-1}, C_1)$ for $n>N_0$. By $a_1, a_2 \in (0, a^{*}]$ and H\"older inequality, there exists a large constant $C_2 > 1$ such that $\| u_{1n}\|_4 + \| u_{2n}\|_4 \in (C_2^{-1}, C_2)$ for $n > N_0$. By \eqref{eq-weak-convergence}, we have $\lim_{n \to \infty} \|u_{in}\|_{4} = \|u_i\|_4$, $i=1,2$, and then $\|u_1\|_4 + \|u_2\|_4>0$. Without loss of generality, we assume $\|u_2\|_4>0$. Denote $t_n = \tfrac{\| u_{1n}\|_4^2}{\| u_{2n}\|_4^2}$. Then $\lim\limits_{n\to \infty} t_n = \tfrac{\| u_{1}\|_4^2}{\| u_{2}\|_4^2}$. We claim that
		\begin{equation}\label{eq-proportion-u1n-u2n}
		\frac{\| u_{1}\|_4^2}{\| u_{2}\|_4^2} = \sqrt{\frac{a^* - a_2}{a^* - a_1}}.
		\end{equation}
Indeed,
\begin{equation*}
\begin{aligned}
&
\mathcal{F}_{a_1, a_2, \beta_{*}}(u_{1n}, u_{2n})
\stackrel{\eqref{eq-GN-inequality}}{\ge}  
a^{*}
\frac{
\| u_{1n} \|_{4}^{4}
+  \| u_{2n} \|_{4}^{4} }{ a_1 \| u_{1n}\|_4^4 + a_2 \| u_{2n} \|_4^4 + 2 \beta_{*} \| u_{1n} u_{2n} \|_2^2}
\\
\ge \ &
a^{*}
\frac{
\| u_{1n} \|_{4}^{4}
+  \| u_{2n} \|_{4}^{4} }{ a_1 \| u_{1n}\|_4^4 + a_2 \| u_{2n} \|_4^4 + 2 \beta_{*} \| u_{1n}\|_4^2 \| u_{2n} \|_4^2}
\\
= \ & 
a^{*}
\frac{
t_n^2
+  1 }{ a_1 t_n^2 + a_2  + 2 \beta_* t_n }
= 
1 +
\frac{
(\sqrt{a^{*} - a_1} t_n -
\sqrt{a^{*} - a_2})^2 }{ a_1 t_n^2 + a_2  + 2 \beta_{*} t_n }.
\end{aligned}
\end{equation*}
If \eqref{eq-proportion-u1n-u2n} fails, it will contradict \eqref{eq-limit}. By \eqref{eq-proportion-u1n-u2n}, we have $\| u_{1}\|_4 > 0$.

By \eqref{eq-weak-convergence} and \eqref{eq-limit}, we have $\lim_{n\to \infty} \| u_{1n} u_{2n} \|_2^2 = \| u_{1} u_{2} \|_2^2$ and
\begin{equation}\label{eq-value-F-u1-u2}
		1 =	\lim_{n \to \infty} \mathcal{F}_{a_1, a_2, \beta_*}(u_{1n}, u_{2n}) = \frac{2}{a_1 \| u_1\|_4^4 + a_2 \| u_2\|_4^4 + 2\beta_* \| u_1 u_2\|_2^2}.
\end{equation}
It follows that
\begin{equation}\label{26July8-3}
\begin{aligned}
& 
2 \stackrel{\eqref{eq-value-F-u1-u2}}{=} a_1 \| u_1\|_4^4 + a_2 \| u_2\|_4^4 + 2\beta_* \| u_1 u_2\|_2^2
\\
\le \ & a_1 \| u_1\|_4^4 + a_2 \| u_2\|_4^4 + 2\beta_*
\| u_1\|_4^2
\| u_2\|_4^2
= 
\Big( 
a_1 \frac{\| u_1\|_4^4}{\| u_2\|_4^4} + 2\beta_*
\frac{\| u_1\|_4^2}
{\| u_2\|_4^2}
+ a_2 
\Big) \| u_2\|_4^4
\\
\stackrel{\eqref{eq-proportion-u1n-u2n}}{=} \ & \Big[ 
a_1 \frac{a^* - a_2}{a^* - a_1} + 2 (a^* - a_2)
+ a_2 
\Big] \| u_2\|_4^4
=
\frac{ a^{*} [ 2 a^{*} - (a_1 + a_2)  ] }{a^{*} - a_1} \| u_2\|_4^4.
\end{aligned}
\end{equation}
Combining \eqref{eq-proportion-u1n-u2n}, we have
\begin{equation}\label{26July8-1}
\| u_1 \|_4^4 \ge \frac{2 (a^* - a_2)}{a^* [2a^* - (a_1 + a_2)]}, \quad 
\| u_2 \|_4^4 \ge \frac{2 (a^* - a_1)}{a^* [2a^* - (a_1 + a_2)]}.
\end{equation}

By \eqref{eq-weak-convergence},
\begin{equation*}
\begin{aligned}
&
\| u_i \|_2^2 \le \liminf_{n \to \infty} \| u_{in}\|_2^2 = 1,
\ i=1,2,
\\
&
\| \nabla u_1 \|_2^2 + \| \nabla u_2 \|_2^2 \le \liminf_{n \to \infty} \| \nabla u_{1n}\|_2^2
+
\liminf_{n \to \infty} \| \nabla u_{2n}\|_2^2
\le 
\liminf_{n\to \infty} 
( 
\| \nabla u_{1n} \|_2^2  + \| \nabla u_{2n} \|_2^2
)
\stackrel{\eqref{thre-26Sep13-4}}{=} 1,
\end{aligned}
\end{equation*}
which implies
\begin{equation*}
\| \nabla u_1\|_2^2 \| u_1\|_2^2 + \| \nabla u_2\|_2^2 \| u_2\|_2^2
\le \| \nabla u_{1} \|_2^2  + \| \nabla u_2\|_2^2
\le 
1.
\end{equation*}
On the other hand,
\begin{equation*}
\| \nabla u_1\|_2^2 \| u_1\|_2^2 + \| \nabla u_2\|_2^2 \| u_2\|_2^2 
\stackrel{\eqref{eq-GN-inequality}}{\ge} \frac{a^*}{2}(\| u_1\|_4^4 + \| u_2\|_4^4) \stackrel{\eqref{26July8-1}}{\ge} 1.
\end{equation*}
By \eqref{eq-proportion-u1n-u2n}, we have $\| \nabla u_i\|_2^2 > 0$, $i=1,2$. Thus, we obtain 
\begin{equation}\label{26July8-4}
\| u_1\|_2^2 = \| u_2\|_2^2 = 1.
\end{equation}
And in \eqref{26July8-3}, by $\beta_*>0$, it holds that $\| u_1 u_2\|_2^2 = \| u_1\|_4^2 \| u_2\|_4^2$, which implies that $|u_1|^2 = C |u_2|^2$ a.e. with a constant $C \ge 0$ by H\"older's inequality \cite[Theorem 2.3]{liebAnalysis2001} . Combining \eqref{26July8-4}, we have
\begin{equation*}
|u_1| = |u_2| \mbox{ \ a.e.},
\end{equation*}
which contradicts \eqref{eq-proportion-u1n-u2n} due to the assumption $a_1 \neq a_2$. Therefore, $\mathcal{O}(a_1, a_2, \beta_*) > 1$.

		\textbf{Case 2: The upper surface $\beta^* = \beta^*(a_1, a_2)$.}
		Choose the trial functions $\big(\frac{Q(x)}{\sqrt{a^*}}, \frac{\lambda Q(\lambda x)}{\sqrt{a^*}}\big)$ with $\lambda \in (0,\infty)$. Using \eqref{eq-Q-property-1}, we obtain
		\begin{equation*}
				\mathcal{F}_{a_1, a_2, \beta^*}\Big(\frac{Q(x)}{\sqrt{a^*}}, \frac{\lambda Q(\lambda x)}{\sqrt{a^*}}\Big) 
				= \frac{(1 + \lambda^2) a^*}{a_1 + a_2 \lambda^2 + \frac{\beta^*}{a^*} \lambda^2 \int_{\R^2} Q^2(x) Q^2(\lambda x) \,\rmd x} =: F(\lambda).
		\end{equation*}        
		Set
		\begin{equation*}
			G(\lambda) := (1 + \lambda^2) a^*, \qquad 
			H(\lambda) := a_1 + a_2 \lambda^2 + \frac{\beta^*}{a^*} \lambda^2 \int_{\R^2} Q^2(x) Q^2(\lambda x) \,\rmd x.
		\end{equation*}
		A straightforward calculation gives $G(1) = H(1) = 2a^*$, which implies $F(1) = 1$. Since $Q, \nabla Q$ have exponentially decay as $|x| \to \infty$, by Green's formula, we compute
		\begin{equation*}
			\begin{aligned}
				H'(1) &= 2a_2 + 4\beta^* + \frac{2\beta^*}{a^*} \int_{\R^2} Q^3(x) \nabla Q(x) \cdot x \,\rmd x 
				= 2a_2 + 4\beta^* - \frac{\beta^*}{a^*} \int_{\R^2} Q^4(x) \,\rmd x \\
				&= 2a_2 + 2\beta^* = 2a^* + a_2 - a_1.
			\end{aligned}
		\end{equation*}
		Thus,
		\begin{equation*}
			F'(1) = \frac{G'(1) H(1) - H'(1) G(1)}{H^2(1)} = \frac{a_1 - a_2}{2a^*} \neq 0.
		\end{equation*}
	This indicates that there exists $\lambda_1$ close to $1$ such that $F(\lambda_1) < F(1) = 1$, which implies $\mathcal{O}(a_1, a_2, \beta^*) < 1$.
	\end{proof}

	From Lemma \ref{lem-existence-1}, to prove Theorem \ref{thm-main}, we only need to prove the following theorem:
	\begin{theorem}\label{thm-gamma}
		Suppose $0 < a_1, a_2 < a^*$, then there exists a unique continuous function $\gamma = \gamma(a_1, a_2)$ satisfying
        $\begin{cases}
        \gamma = \beta_{*} = \beta^{*}
        & \mbox{ if \ } a_1 = a_2
        \\
        \gamma \in (\beta_{*}, \beta^{*}) & \mbox{ if \ } a_1 \ne a_2
        \end{cases} $ such that
\begin{equation}\label{26July9-2}
			\mathcal{O}(a_1, a_2, \beta) \begin{cases}
				> 1 & \mbox{ \ if \ } 0 < \beta < \gamma
                \\
				= 1 & \mbox{ \ if \ } \beta = \gamma
                \\
				< 1 & \mbox{ \ if \ } \beta > \gamma.
			\end{cases}
		\end{equation}
	\end{theorem}

	\begin{proof}[Proof of Theorem \ref{thm-gamma}]\textbf{Existence and uniqueness of critical surface $\gamma(a_1,a_2)$ satisfying \eqref{26July9-2}.} 
    Given $a_1, a_2$, by Lemma \ref{lem-O-properties} (iv), if there exists a $\gamma$ satisfying \eqref{26July9-2}, then $\gamma$ must be unique.

    When $a_1 = a_2$, we take $\gamma = \beta_{*} = \beta^{*}$. Then, Lemma \ref{lem-O-properties} (ii), (iii), and Remark \ref{26July9-1-rmk} (c) deduce Theorem \ref{thm-gamma}.

When $a_1 \ne a_2$, $\beta_{*} < \beta^{*}$. By Lemma \ref{lem-O-energy-estimate}, we have $\mathcal{O}(a_1, a_2, \beta_*) > 1$ and $\mathcal{O}(a_1, a_2, \beta^*) < 1$. Set
\begin{equation}
\gamma = \gamma(a_1, a_2):=\min\{ \beta  \mid \mathcal{O}(a_1,a_2,\beta)=1\}.
\end{equation}
By Lemma \ref{lem-O-properties} (i), (iv),
\begin{equation*}
\gamma \in (\beta_{*}, \beta^{*}),
\quad
\mathcal{O}(a_1, a_2, \gamma)=1,
\mbox{ \ and \ } 
\mathcal{O}(a_1, a_2, \beta) > 1 \mbox{ \ for \ } 0<\beta < \gamma.
\end{equation*}       
Since $\mathcal{O}(a_1, a_2, \gamma) = 1 < \min\{\frac{a^*}{a_1}, \frac{a^*}{a_2} \}$, by Corollary \ref{cor-decrease},
\begin{equation*}
\mathcal{O}(a_1, a_2, \beta)  <  1 \mbox{ \ for \ } \beta > \gamma.
\end{equation*}

		\textbf{Continuity of $\gamma(a_1, a_2)$.}
		Given any $0< a_1^0, a_2^0 < a^{*}$, denote $\gamma^0 = \gamma(a_1^0, a_2^0)$. For any $\varepsilon > 0$, we need to show that there exists $\delta > 0$ such that 
        \begin{equation}\label{26July9-6}
        |\gamma(a_1, a_2) - \gamma^0| < \varepsilon \mbox{ \ for \ } |a_1 - a_1^0| + |a_2 - a_2^0| < \delta,
        \ 
        0 < a_1, a_2 < a^*.
        \end{equation}
	By \eqref{26July9-2},
		\begin{equation*}
        \mathcal{O}(a_1^0, a_2^0, \gamma^0) = 1,
        \quad
			\mathcal{O}(a_1^0, a_2^0, \gamma^0 - \varepsilon) > 1 > \mathcal{O}(a_1^0, a_2^0, \gamma^0 + \varepsilon).
		\end{equation*}
By Lemma \ref{lem-O-properties} (i),  there exists $\delta_1 > 0$ such that for $(a_1, a_2)$ satisfying $|a_1 - a_1^0| + |a_2 - a_2^0| < \delta_1$, $0 < a_1, a_2 < a^*$, we have
		\begin{equation*}
			\big| \mathcal{O}(a_1, a_2, \gamma^0 - \varepsilon) - \mathcal{O}(a_1^0, a_2^0, \gamma^0 - \varepsilon) \big| < [\mathcal{O}(a_1^0, a_2^0, \gamma^0 - \varepsilon) - 1]/2.
		\end{equation*}
		Consequently,
		\begin{equation*}
			\mathcal{O}(a_1, a_2, \gamma^0 - \varepsilon) > 1.
		\end{equation*}
		Similarly, there exists $\delta_2 > 0$ such that
		\begin{equation*}
			\mathcal{O}(a_1, a_2, \gamma^0 + \varepsilon) < 1
            \mbox{ \ for \ } 
            |a_1 - a_1^0| + |a_2 - a_2^0| < \delta_2, \ 
            0 < a_1, a_2 < a^*.
		\end{equation*}
Take $\delta = \min\{\delta_1, \delta_2\}$. Then
		\begin{equation*}
			\mathcal{O}(a_1, a_2, \gamma^0 - \varepsilon) > 1 > \mathcal{O}(a_1, a_2, \gamma^0 + \varepsilon)
            \mbox{ \ for \ } 
            |a_1 - a_1^0| + |a_2 - a_2^0| < \delta, 
            \ 
            0 < a_1, a_2 < a^*.
		\end{equation*}
Lemma \ref{lem-O-properties} (iv) and $\mathcal{O}(a_1, a_2, \gamma(a_1, a_2)) = 1$ deduce $\gamma^0 - \varepsilon < \gamma(a_1,a_2) < \gamma^0 + \varepsilon$. Then \eqref{26July9-6} holds.
\end{proof}

\begin{proof}[Proof of Theorem \ref{thm-main}]
	The result follows directly from Theorem \ref{thm-gamma} and Lemma \ref{lem-existence-1}.
\end{proof}

Finally, we focus on the proof of Theorem \ref{thm-main-2}.
Inspired by \cite{guoBlowupSolutionsTwo2017a}, the key step is to prove that $\mathfrak{e}(a_1, a_2, \gamma) = 1$ given in \eqref{eq-minimizer-problem-e}. 

\begin{proof}[Proof of Theorem \ref{thm-main-2}]

Since $0 < a_1, a_2 < a^*$, by Theorem \ref{thm-gamma}, $\mathcal{O}(a_1,a_2,\gamma) = 1$. By $V_1, V_2 \ge 0$, $\mathcal{O}(a_1, a_2, \gamma) = 1$, and the relationship between  $\mathcal{E}_{a_1, a_2, \gamma}(\cdot)$ in \eqref{th-26Sep18-4} and $\mathcal{F}_{a_1,a_2, \gamma}(\cdot)$, we have $\mathfrak{e}(a_1,a_2,\gamma) \ge 0$.

Next, we prove $\mathfrak{e}(a_1,a_2,\gamma) \le 0$.  
By $0 < a_1, a_2 < a^*$ and Lemma \ref{lem-existence-O}, there exist $u_1, u_2 \in H^1(\mathbb{R}^2)$ attaining $\mathcal{O}(a_1,a_2,\gamma) = 1$ such that
    \begin{equation*}
        \mathcal{F}_{a_1,a_2, \gamma}(u_1, u_2) = 1,
        \quad
        \| u_1\|_2^2 = \| u_2\|_2^2 = 1.
    \end{equation*}
Moreover, $u_1, u_2$ are non-negative, smooth, radially symmetric, and 
\begin{equation}\label{th-26Sep18-2}
|u_i| + |\nabla u_i| \le C \rme^{- \sigma |x|}
\end{equation}
with some constants $C, \sigma > 0$ in $\mathbb{R}^2$, $i=1,2$. Given a smooth cut-off function $\eta$ satisfying
    \begin{equation*}
         0 \le \eta \le1 \text{ \ in \ } \mathbb{R}^2,
        \quad
        \eta=1 \text{ \ in \ }B_1(0),
        \quad
        \eta=0 \text{ \ in \  } \R^2\setminus B_2(0),
    \end{equation*}
for any $\tau \ge 1$, $i=1,2$, we define
    \begin{equation*}
        u_{i,\tau}(x)
    :=
    c_{i,\tau}\tau
    \eta(x-x_0)
    u_i( \tau(x-x_0) )
    \end{equation*}
with a constant $c_{i,\tau}>0$ to make $\| u_{i,\tau}\|_2^2=1$. Direct calculation gives
\begin{equation*}
\| u_{i,\tau}\|_2^2
=
c_{i,\tau}^2 \int_{\mathbb{R}^2} | \eta( \tau^{-1} y ) u_i(y) |^2  \rmd y.
\end{equation*}
Since $\|u_i\|_{2} = 1$, $|u_i| \lesssim \rme^{-\sigma |x|}$, we have 
\begin{equation}\label{th-26Sep18-1}
\int_{\mathbb{R}^2} | \eta( \tau^{-1} y ) u_i(y) |^2  \rmd y = 
1 + O(\rme^{-\sigma \tau})
\mbox{ \ and then \ }
c_{i,\tau} = \Big( \int_{\mathbb{R}^2} | \eta( \tau^{-1} y ) u_i(y) |^2  \rmd y \Big)^{-1/2} = 1 + O(\rme^{-\sigma \tau}).
\end{equation}
By \eqref{th-26Sep18-2}, a direct calculation gives
\begin{equation}\label{th-26Sep18-6}
    \| \nabla u_{i,\tau}\|_2^2
    =
    \tau^2\| \nabla u_i\|_2^2+ O(\rme^{-\epsilon \tau}),
    \quad
    \| u_{i,\tau}\|_4^4
    =
    \tau^2\| u_i\|_4^4+ O(\rme^{-\epsilon \tau}),
    \quad
    \| u_{1,\tau} u_{2,\tau}\|_2^2
    =
    \tau^2\| u_1 u_2\|_2^2+ O(\rme^{-\epsilon \tau})
\end{equation}
with a sufficiently small constant $\epsilon > 0$. Indeed, for the third part,
\begin{equation*}
\begin{aligned}
&
\| u_{1,\tau} u_{2,\tau}\|_2^2
=
c_{1,\tau}^2 c_{2,\tau}^2 \tau^4 
\int_{\mathbb{R}^2}
\eta^4(x-x_0)
    u_1^2\big(\tau(x-x_0)\big)
    u_2^2\big(\tau(x-x_0)\big) \rmd x
\\
= \ & c_{1,\tau}^2 c_{2,\tau}^2 \tau^2 
\int_{\mathbb{R}^2}
\eta^4(\tau^{-1} y)
    u_1^2(y)
    u_2^2(y) \rmd y
=  \tau^2  
\Big( \int_{\mathbb{R}^2}
u_1^2(y) u_2^2(y) \rmd y + O(\rme^{-\sigma \tau}) \Big).
\end{aligned}
\end{equation*}
The first and second parts can be deduced similarly. Additionally,
\begin{equation*}
\begin{aligned}
&
\int_{\mathbb{R}^2}V_i(x)|u_{i,\tau}(x)|^2 \rmd x
= c_{i,\tau}^2 \tau^2 \int_{\mathbb{R}^2} V_i(x) \eta^2(x-x_0) u_i^2\big(\tau(x-x_0)\big)
\rmd x
\\
= \ & c_{i,\tau}^2 \int_{\mathbb{R}^2} V_i(x_0 + \tau^{-1} y) \eta^2(\tau^{-1} y) u_i^2(y)
\rmd y.
\end{aligned}
\end{equation*}
Since $V_i(x_0) = 0$ and $V_i(x)$ is continuous at $x_0$, for any $y \in \mathbb{R}^2$, we have
\begin{equation*}
\lim_{\tau\to\infty} V_i(x_0 + \tau^{-1} y) \eta^2(\tau^{-1} y) u_i^2(y) = 0.
\end{equation*}
By $V_i \in L_{\rm{loc}}^{\infty}(\mathbb{R}^2)$ and the cut-off function $\eta$, there exists a constant $C_1>0$ such that
\begin{equation*}
|V_i(x_0 + \tau^{-1} y) \eta^2(\tau^{-1} y)| \le C_1 \mbox{ \ for \ } y\in \mathbb{R}^2, \ \tau \ge 1. 
\end{equation*}
By Lebesgue's dominated convergence theorem and \eqref{th-26Sep18-1}, we have
\begin{equation*}
\lim_{\tau\to \infty} \int_{\mathbb{R}^2}V_i(x)|u_{i,\tau}(x)|^2 \rmd x = 0.
\end{equation*}
Hence, we have $\{(u_{1,\tau}, u_{2,\tau})\}_{\tau \ge 1} \subset \mathcal{M}$ defined in \eqref{thre-26Sep12-1}. For $\mathcal{E}_{a_1, a_2, \gamma}(\cdot)$ defined in \eqref{th-26Sep18-4}, by \eqref{th-26Sep18-6} and $\mathcal{F}_{a_1,a_2, \gamma}(u_1, u_2) = 1$, we have
\begin{equation*}
\mathcal{E}_{a_1, a_2, \gamma}(u_{1,\tau}, u_{2,\tau})
= 
O(\rme^{-\epsilon \tau}) + \sum_{i=1}^{2} \int_{\R^2} V_i(x) |u_{i\tau}(x)|^2 \rmd x,
\end{equation*}
which goes to $0$ as $\tau \to \infty$. Thus, $\mathfrak{e}(a_1,a_2,\gamma) \le 0$. In sum, we get $\mathfrak{e}(a_1,a_2,\gamma)=0$.

If $\mathfrak e(a_1,a_2,\gamma)=0$ is attained by some $(w_1,w_2)\in\mathcal{M}$, since $\mathcal{O}(a_1,a_2,\gamma) = 1$, then we would have $\int_{\mathbb R^2} V_i |w_i|^2 \rmd x=0$ for $i=1,2$, and $\mathcal F_{a_1,a_2,\gamma}(w_1,w_2)=1$.
	Since \(V_i(x)\to\infty\) as \(|x|\to\infty\), there exists \(R>0\) sufficiently large such that \(V_i(x)>1\) in $|x| \ge R$. Hence $w_i=0$ a.e. in $|x| \ge R$.
	Moreover, since $\mathcal O(a_1,a_2,\gamma)=1$, $(w_1,w_2)$ is a minimizer of $ \mathcal O(a_1,a_2,\gamma)$. Similar to \eqref{th-26Sep17-1}, $(w_1, w_2)$ is a weak
	solution of the Schr\"odinger-type system
    \begin{equation}
    \begin{cases}
        -\Delta w_1 + \lambda_1 w_1 = a_1 w_1^3 + \gamma w_1 w_2^2, 
        \\
        -\Delta w_2 + \lambda_2 w_2 = a_2 w_2^3 + \gamma w_1^2 w_2,
    \end{cases}
        \mbox{ \ in \ } \mathbb{R}^2
    \end{equation} 
with some constants $\lambda_1, \lambda_2$. By the strong unique continuation property for Schr\"odinger equations (see \cite[Remark 6.7]{509f7a07-890c-3be6-998a-aca39a278594}),  we obtain $ w_i\equiv0$, which contradicts $\|w_i\|_2^2=1$.
\end{proof}


\end{document}